\documentclass[12pt,reqno]{amsart}
\usepackage{amsmath, amsfonts, amssymb, amsthm, hyperref}
\usepackage{bm}
\allowdisplaybreaks[4]
\def\l{\left}
\def\r{\right}
\def\bg{\bigg}
\def\({\bg(}
\def\){\bg)}
\def\t{\text}
\def\f{\frac}

\def\eq{\equiv}

\def\N{\mathbb N}

\def\1{{\bf 1}}

\theoremstyle{plain}
\newtheorem{theorem}{Theorem}[section]
\newtheorem{lemma}{Lemma}

\theoremstyle{definition}

\newtheorem*{Acks}{Acknowledgments}
\theoremstyle{remark}
\newtheorem{remark}{Remark}

\def\<{\langle}
\def\>{\rangle}

\begin{document}
\hbox{}
\medskip

\title[On two conjectural supercongruences of Z.-W. Sun]{On two conjectural supercongruences of Z.-W. Sun}

\author{Chen Wang}
\address {Department of Mathematics, Nanjing
University, Nanjing 210093, People's Republic of China}
\email{cwang@smail.nju.edu.cn}

\subjclass[2020]{Primary 33C20, 11A07; Secondary 11B65, 05A10}
\keywords{supercongruences, hypergeometric series, binomial coefficients}

\begin{abstract}
In this paper, we mainly prove two conjectural supercongruences of Sun by using the following identity
$$
\sum_{k=0}^n\binom{2k}{k}^2\binom{2n-2k}{n-k}^2=16^n\sum_{k=0}^n\frac{\binom{n+k}{k}\binom{n}{k}\binom{2k}{k}^2}{(-16)^k}
$$
which arises from a ${}_4F_3$ hypergeometric transformation. For any prime $p>3$, we prove that
\begin{gather*}
\sum_{n=0}^{p-1}\frac{n+1}{8^n}\sum_{k=0}^n\binom{2k}{k}^2\binom{2n-2k}{n-k}^2\equiv(-1)^{(p-1)/2}p+5p^3E_{p-3}\pmod{p^4},\\
\sum_{n=0}^{p-1}\frac{2n+1}{(-16)^n}\sum_{k=0}^n\binom{2k}{k}^2\binom{2n-2k}{n-k}^2\equiv(-1)^{(p-1)/2}p+3p^3E_{p-3}\pmod{p^4},
\end{gather*}
where $E_{p-3}$ is the $(p-3)$th Euler number.
\end{abstract}
\maketitle

\section{Introduction}
\setcounter{lemma}{0}
\setcounter{theorem}{0}
\setcounter{equation}{0}
\setcounter{conjecture}{0}
\setcounter{remark}{0}

The truncated hypergeometric series are defined by
$$
{}_nF_{n-1}\bigg[\begin{matrix}x_1&x_2&\cdots&x_n\\ &y_1&\cdots&y_{n-1}\end{matrix}\bigg|\ z\bigg]_m=\sum_{k=0}^m\f{(x_1)_k(x_2)_k\cdots(x_n)_k}{(y_1)_k(y_2)_k\cdots(y_{n-1})_k}\f{z^k}{k!},
$$
where
$$
(x)_k=\begin{cases}1,\quad &k=0,\\
x(x+1)\cdots(x+k-1),\quad &k>0.\end{cases}
$$
denotes the so-called Pochhammer symbol (or rising factorial). Clearly, they are truncations of the classical hypergeometric series. Since $(-x)_k/(1)_k=(-1)^k\binom{x}{k}$, sometimes we may write the truncated hypergeometric series as sums involving products of binomial coefficients. In recent years, there is a rising interest in studying supercongruences involving truncated hypergeometric series (cf., for example, \cite{Mao,MP,RV,SunZH14,Sun11,Sun2012acta,vanHamme97}).

In 2003, Rodriguez-Villegas \cite{RV} studied hypergeometric families of Calabi-Yau manifolds and discovered (numerically) 22 supercongruences concerning truncated hypergeometric series. For example, he conjectured that for any odd prime $p$,
\begin{equation}\label{RVcon}
\sum_{k=0}^{p-1}\f{\binom{2k}{k}^2}{16^k}\eq(-1)^{(p-1)/2}\pmod{p^2},
\end{equation}
which was later confirmed by Mortenson \cite{Mortenson03} using Gaussian hypergeometric series and Gross-Koblitz formula (see \cite{R} for details about Gross-Koblitz formula). Quite recently, Barman and Saikia \cite{BS} obtained a parametric generalization of \eqref{RVcon} without using Gaussian hypergeometric series. Note that $\binom{2k}{k}\eq0\pmod{p}$ for $k\in\{(p+1)/2,\ldots,p-1\}$. Thus
$$
\sum_{k=0}^{p-1}\f{\binom{2k}{k}^2}{16^k}\eq\sum_{k=0}^{(p-1)/2}\f{\binom{2k}{k}^2}{16^k}\pmod{p^2}.
$$
For more parametric generalizations of \eqref{RVcon}, the reader may consult \cite{Guo,GPZ,GZ,Liu,MP,NP,SunZH14}.

Recall that the Euler numbers $E_n$ ($n\in\N$) are defined by
$$
E_0=1,\ \t{and}\ \sum_{\substack{k=0\\ 2|k}}^n\binom{n}{k}E_{n-k}=0\ \t{for}\ n=1,2,\ldots.
$$

In 2011, Sun \cite{Sun11} investigated some congruences related to the Euler numbers. Especially, for any prime $p>3$ he proved the following two congruences as extensions of \eqref{RVcon}:
\begin{equation}\label{suncon1}
\sum_{k=0}^{(p-1)/2}\f{\binom{2k}{k}^2}{16^k}\eq(-1)^{(p-1)/2}+p^2E_{p-3}\pmod{p^3}
\end{equation}
and
\begin{equation}\label{suncon2}
\sum_{k=(p+1)/2}^{p-1}\f{\binom{2k}{k}^2}{16^k}\eq-2p^2E_{p-3}\pmod{p^3}.
\end{equation}
In \cite{Sun11}, Sun also conjectured many congruences most of which have been confirmed. One of them is as follows: for any prime $p>3$,
\begin{equation}\label{CXHcon}
\sum_{k=0}^{p-1}\f{3k+1}{(-8)^k}\binom{2k}{k}^3\eq (-1)^{(p-1)/2}p+p^3E_{p-3}\pmod{p^4}.
\end{equation}
This was confirmed by Chen, Xie and He \cite{CXH} in 2016. We also note that for any prime $p>3$ Mao \cite{Mao} showed that
$$
\sum_{k=0}^{(p-1)/2}\f{3k+1}{(-8)^k}\binom{2k}{k}^3\eq (-1)^{(p-1)/2}p+\f{(-1)^{(p^2-1)/8}p^3}{4}E_{p-3}\l(\f14\r)\pmod{p^4},
$$
where the Euler polynomials $E_n(x)$ ($n\in\N$) are given by
$$
E_n(x)=\sum_{k=0}^{n}\binom{n}{k}\f{E_k}{2^k}\l(x-\f12\r)^{n-k}.
$$

In 2012, Sun \cite{Sun2012acta} studied congruences for sums involving products of three binomial coefficients systematically. Recall that for any prime $p\eq1\pmod{4}$, we may write $p=x^2+y^2$ with $x\eq1\pmod{4}$ and $y\eq0\pmod{2}$. In \cite{Sun2012acta}, Sun determined $x\pmod{p^2}$ as follows:
\begin{equation}\label{sunx}
(-1)^{(p-1)/4}x\eq\sum_{k=0}^{(p-1)/2}\f{k+1}{8^k}\binom{2k}{k}^2\eq\sum_{k=0}^{(p-1)/2}\f{2k+1}{(-16)^k}\binom{2k}{k}^2\pmod{p^2}.
\end{equation}
In the proof of \eqref{sunx}, Sun also obtained the following congruences:
\begin{equation}\label{sun8^np^3}
\sum_{n=0}^{p-1}\f{n+1}{8^n}\sum_{k=0}^{n}\binom{2k}{k}^2\binom{2n-2k}{n-k}^2\eq (-1)^{(p-1)/2}p\pmod{p^3}
\end{equation}
and
\begin{equation}\label{sun-16^np^3}
\sum_{n=0}^{p-1}\f{2n+1}{(-16)^n}\sum_{k=0}^{n}\binom{2k}{k}^2\binom{2n-2k}{n-k}^2\eq (-1)^{(p-1)/2}p\pmod{p^3}
\end{equation}
for any odd prime $p$.

The main goal of this paper is to establish the following generalizations of \eqref{sun8^np^3} and \eqref{sun-16^np^3} which ware conjectured by Sun (see both \cite[Conjecture 4.1]{Sun2014pi} and \cite[Conjecture 33(ii)]{SunZW19}).

\begin{theorem}\label{maintheorem1}
For any prime $p>3$, we have
\begin{gather}
\label{den8^n}\sum_{n=0}^{p-1}\f{n+1}{8^n}\sum_{k=0}^n\binom{2k}{k}^2\binom{2n-2k}{n-k}^2\eq(-1)^{(p-1)/2}p+5p^3E_{p-3}\pmod{p^4},\\
\label{den-16^n}\sum_{n=0}^{p-1}\f{2n+1}{(-16)^n}\sum_{k=0}^n\binom{2k}{k}^2\binom{2n-2k}{n-k}^2\eq(-1)^{(p-1)/2}p+3p^3E_{p-3}\pmod{p^4}.
\end{gather}
\end{theorem}

\begin{remark}
In \cite[Conjecture 4.1]{Sun2014pi}, Sun also conjectured that
$$
\sum_{n=0}^{p-1}\f{n}{32^n}\sum_{k=0}^n\binom{2k}{k}^2\binom{2n-2k}{n-k}^2\eq-2p^3E_{p-3}\pmod{p^4}.
$$
This was confirmed by Mao and Cao \cite[Theorem 2.1]{MC} recently.
\end{remark}

In next section, we shall first prove \eqref{den-16^n} by establishing a new transformation for the summation
$
\sum_{k=0}^{n}\binom{2k}{k}^2\binom{2n-2k}{n-k}^2.
$
Then via a hypergeometric transformation due to Chaundy and Bullard we will show that \eqref{den8^n} is actually a corollary of \eqref{den-16^n}.

\medskip
\section{Proof of Theorem \ref{maintheorem1}}
\setcounter{lemma}{0}
\setcounter{theorem}{0}
\setcounter{equation}{0}
\setcounter{conjecture}{0}
\setcounter{remark}{0}

In order to show \eqref{den-16^n} we need the following lemmas.

\begin{lemma}\label{keyid1}
Let $n$ be a nonnegative integer. Then we have
\begin{equation}\label{id1}
\sum_{k=0}^n\binom{2k}{k}^2\binom{2n-2k}{n-k}^2=16^n\sum_{k=0}^n\f{\binom{n+k}{k}\binom{n}{k}\binom{2k}{k}^2}{(-16)^k}.
\end{equation}
\end{lemma}

\begin{proof}
It is easy to check that
$$
\binom{2k}{k}\binom{2n-2k}{n-k}=4^n\f{(\f12)_k(\f12)_{n-k}}{(1)_k(1)_{n-k}}
$$
and
$$
\f{(\f12)_{n-k}}{(1)_{n-k}}=\f{(\f12)_n(-n)_k}{(1)_n(\f12-n)_k}=\f{\binom{2n}{n}(-n)_k}{4^n(\f12-n)_k}.
$$
Hence we obtain
\begin{equation}\label{sumtohyper}
\sum_{k=0}^{n}\binom{2k}{k}^2\binom{2n-2k}{n-k}^2=\binom{2n}{n}^2{}_4F_3\bigg[\begin{matrix}\f12&\f12&-n&-n\\ &1&\f12-n&\f12-n\end{matrix}\bigg|\ 1\bigg],
\end{equation}
here we note that the hypergeometric series in the right-hand side is actually a finite sum since $(-n)_k=0$ for all $k>n$.

It is known from \cite[Theorem 3.3.3]{AAR} that
\begin{equation}\label{knownid}
{}_4F_3\bigg[\begin{matrix}-n&a&b&c\\ &d&e&f\end{matrix}\bigg|\ 1\bigg]=\f{(e-a)_n(f-a)_n}{(e)_n(f)_n}{}_4F_3\bigg[\begin{matrix}-n&a&d-b&d-c\\ &d&a+1-n-e&a+1-n-f\end{matrix}\bigg|\ 1\bigg]
\end{equation}
provided that $a+b+c-n+1=d+e+f$. Letting $a=c=1/2,\ b=-n,\ d=1,\ e=f=1/2-n$ in \eqref{knownid} we arrive at
\begin{align}\label{4F3transform}
{}_4F_3\bigg[\begin{matrix}\f12&\f12&-n&-n\\ &1&\f12-n&\f12-n\end{matrix}\bigg|\ 1\bigg]=&\f{(-n)_n^2}{(\f12-n)_n^2}{}_4F_3\bigg[\begin{matrix}-n&n+1&\f12&\f12\\&1&1&1\end{matrix}\bigg|\ 1\bigg]\notag\\
=&\f{16^n}{\binom{2n}{n}^2}\sum_{k=0}^{n}\f{\binom{n+k}{k}\binom{n}{k}\binom{2k}{k}^2}{(-16)^k}.
\end{align}
Now substituting \eqref{4F3transform} into \eqref{sumtohyper} we immediately obtain the desired \eqref{id1}.
\end{proof}

\begin{remark} Note that in \cite[Lemma 3.1]{Sun2012acta} Sun obtained another transformation of the summation $\sum_{k=0}^n\binom{2k}{k}^2\binom{2n-2k}{n-k}^2$ as follows:
\begin{equation}\label{suntransform}
\sum_{k=0}^n\binom{2k}{k}^2\binom{2n-2k}{n-k}^2=\sum_{k=0}^n\binom{2k}{k}^3\binom{k}{n-k}(-16)^{n-k},
\end{equation}
and he used \eqref{suntransform} to prove \eqref{sun8^np^3} and \eqref{sun-16^np^3}. We attempted to prove \eqref{den-16^n} by \eqref{suntransform} but failed. However, this transformation is useful for proving a congruence relation between \eqref{den8^n} and \eqref{den-16^n}.
\end{remark}

\begin{lemma}\label{keyid2}
For nonnegative integers $k$ and $l$ with $l\geq k$, we have
\begin{equation}\label{id2}
\sum_{n=k}^l(-1)^n(2n+1)\binom{n+k}{2k}=(-1)^l(l-k+1)\binom{l+k+1}{2k}.
\end{equation}
\end{lemma}

\begin{proof}
It can be verified directly by induction on $l$.
\end{proof}

\medskip
\noindent{\it Proof of \eqref{den-16^n}}. In view of Lemmas \ref{keyid1} and \ref{keyid2}, we have
\begin{align}\label{key1}
&\sum_{n=0}^{p-1}\f{2n+1}{(-16)^n}\sum_{k=0}^{n}\binom{2k}{k}^2\binom{2n-2k}{n-k}^2=\sum_{n=0}^{p-1}(-1)^n(2n+1)\sum_{k=0}^n\f{\binom{n+k}{k}\binom{n}{k}\binom{2k}{k}^2}{(-16)^k}\notag\\
=&\sum_{k=0}^{p-1}\f{\binom{2k}{k}^3}{(-16)^k}\sum_{n=k}^{p-1}(-1)^n(2n+1)\binom{n+k}{2k}=\sum_{k=0}^{p-1}\f{\binom{2k}{k}^3}{(-16)^k}(p-k)\binom{p+k}{2k}\notag\\
=&p\sum_{k=0}^{p-1}\f{\binom{2k}{k}^2\binom{p-1}{k}\binom{p+k}{k}}{(-16)^k}\notag\\
\eq&p\sum_{k=0}^{p-1}\f{\binom{2k}{k}^2}{16^k}(1-p^2H_k^{(2)})\pmod{p^5},
\end{align}
where $H_k^{(2)}=\sum_{j=1}^k1/j^2$ denotes the $k$th harmonic number of order $2$ and the last step follows from the fact
$$
\binom{p-1}{k}\binom{p+k}{k}=(-1)^k\prod_{j=1}^k\l(1-\f{p^2}{j^2}\r)\eq(-1)^k(1-p^2H_k^{(2)})\pmod{p^4}
$$
for $k$ among $0,1,\ldots,p-1$.

In 2015, Sun \cite[Theorem 4.1]{Sun2015pi} obtained that
\begin{equation}\label{suncon3}
\sum_{k=0}^{(p-1)/2}\f{\binom{2k}{k}^2}{16^k}H_k^{(2)}\eq\sum_{k=0}^{p-1}\f{\binom{2k}{k}^2}{16^k}H_k^{(2)}\eq-4E_{p-3}\pmod{p}
\end{equation}
for any prime $p>3$.

Substituting \eqref{suncon1}, \eqref{suncon2} and \eqref{suncon3} into \eqref{key1} we finally obtain \eqref{den-16^n}.\qed

\medskip

To show \eqref{den8^n} we need the following preliminary results.

\begin{lemma}\label{keyid3}
For any nonnegative integer $k$, we have
\begin{equation}\label{id3}
\sum_{n=0}^k\binom{n+k}{n}2^n=(-1)^{k+1}-(-2)^{k+1}\sum_{n=0}^k\binom{n+k}{n}(-1)^n.
\end{equation}
\end{lemma}

\begin{remark}
This is a corollary of the following identity due to Chaundy and Bullard \cite{CB}:
$$
1=(1-x)^{n+1}\sum_{k=0}^m\binom{n+k}{k}x^k+x^{m+1}\sum_{k=0}^n\binom{m+k}{k}(1-x)^k.
$$
\end{remark}

\begin{lemma}\label{keyid4}
For any positive integer $k$ we have
\begin{equation}\label{id4}
\sum_{n=0}^{k-1}(2n+k)\binom{-k}{n}=\f{(-1)^{k-1}k}{2}\binom{2k}{k}.
\end{equation}
\end{lemma}

\begin{proof}
Clearly,
\begin{align*}
&\sum_{n=0}^{k-1}(2n+k)\binom{-k}{n}=\sum_{n=0}^{k-1}(n+k)\binom{-k}{n}+\sum_{n=0}^{k-1}n\binom{-k}{n}\\
=&k\sum_{n=0}^{k-1}\binom{-k-1}{n}-k\sum_{n=0}^{k-2}\binom{-k-1}{n}=k\binom{-k-1}{k-1}=\f{(-1)^{k-1}k}{2}\binom{2k}{k}.
\end{align*}
This concludes the proof.
\end{proof}

\begin{lemma}\label{keyid5}
For any positive integer $k$, we have
\begin{equation}\label{id5}
\sum_{n=0}^{k-1}(-2)^n(n-k+1)\binom{-k}{n}=(-1)^{k+1}(3k-1)-(-2)^k\sum_{n=0}^{k-1}(2n-2k+1)\binom{-k}{n}.
\end{equation}
\end{lemma}

\begin{proof}
Note that $\binom{-k}{n}=(-1)^n\binom{n+k-1}{n}$. Thus by Lemma \ref{keyid3} we have
\begin{align*}
&\sum_{n=0}^{k-1}(-2)^n(n-k+1)\binom{-k}{n}=-k\sum_{n=1}^{k-1}(-2)^n\binom{-k-1}{n-1}-(k-1)\sum_{n=0}^{k-1}\binom{n+k-1}{n}2^n\\
=&2k\sum_{n=0}^k\binom{n+k}{n}2^n-(k-1)\sum_{n=0}^{k-1}\binom{n+k-1}{n}2^n-5k\binom{2k}{k}2^{k-1}\\
=&(-1)^{k+1}(3k-1)-2k(-2)^{k+1}\sum_{n=0}^k\binom{n+k}{n}(-1)^n+(k-1)(-2)^k\sum_{n=0}^{k-1}\binom{n+k-1}{n}(-1)^n\\
&-5k\binom{2k}{k}2^{k-1}\\
=&(-1)^{k+1}(3k-1)+(-2)^k\sum_{n=0}^{k-1}(4n+5k-1)\binom{-k}{n}+3k\binom{2k}{k}2^{k-1}.
\end{align*}
Now by Lemma \ref{keyid4} we have
\begin{align*}
\sum_{n=0}^{k-1}(4n+5k-1)\binom{-k}{n}=&\sum_{n=0}^{k-1}(6n+3k)\binom{-k}{n}-\sum_{n=0}^{k-1}(2n-2k+1)\binom{-k}{n}\\
=&\f{3(-1)^{k-1}k}{2}\binom{2k}{k}-\sum_{n=0}^{k-1}(2n-2k+1)\binom{-k}{n}.
\end{align*}

Combining the above, we finally obtain \eqref{id5}.
\end{proof}

\begin{lemma}\label{keyid6}
For any nonnegative integer $k$, we have the following identities:
\begin{gather}
\label{id6}\sum_{n=0}^k(-2)^n(n+k+1)\binom{k}{n}=(-1)^k(3k+1),\\
\label{id7}\sum_{n=0}^{k}(2n+2k+1)\binom{k}{n}=2^k(3k+1).
\end{gather}
\end{lemma}

\begin{proof}
These two identities can be easily deduced by binomial theorem. Here we just prove \eqref{id6} as an example. It is clear that
\begin{align*}
&\sum_{n=0}^k(-2)^n(n+k+1)\binom{k}{n}=(-1)^k(k+1)+k\sum_{n=1}^k(-2)^n\binom{k-1}{n-1}\\
=&(-1)^k(k+1)-2k\sum_{n=0}^{k-1}(-2)^n\binom{k-1}{n}=(-1)^k(3k+1).
\end{align*}
\end{proof}

\medskip
\noindent{\it Proof of \eqref{den8^n}}. By \eqref{suntransform} and Lemma \ref{keyid6}, we have
\begin{align}\label{key8^n}
&\sum_{n=0}^{p-1}\f{n+1}{8^n}\sum_{k=0}^{n}\binom{2k}{k}^2\binom{2n-2k}{n-k}^2=\sum_{k=0}^{p-1}\f{\binom{2k}{k}^3}{(-16)^k}\sum_{n=k}^{p-1}(-2)^n(n+1)\binom{k}{n-k}\notag\\
=&\sum_{k=0}^{p-1}\f{\binom{2k}{k}^3}{8^k}\sum_{n=0}^{p-1-k}(-2)^n(n+k+1)\binom{k}{n}\notag\\
=&\sum_{k=0}^{(p-1)/2}(3k+1)\f{\binom{2k}{k}^3}{(-8)^k}+\sum_{k=(p+1)/2}^{p-1}\f{\binom{2k}{k}^3}{8^k}\sum_{n=0}^{p-1-k}(-2)^n(n+k+1)\binom{k}{n}\notag\\
=&\sum_{k=0}^{(p-1)/2}(3k+1)\f{\binom{2k}{k}^3}{(-8)^k}+\sum_{k=1}^{(p-1)/2}\f{\binom{2p-2k}{p-k}^3}{8^{p-k}}\sum_{n=0}^{k-1}(-2)^n(n+p-k+1)\binom{p-k}{n}\notag\\
\eq&\sum_{k=0}^{(p-1)/2}(3k+1)\f{\binom{2k}{k}^3}{(-8)^k}+\sum_{k=1}^{(p-1)/2}\f{\binom{2p-2k}{p-k}^3}{8^{p-k}}\sum_{n=0}^{k-1}(-2)^n(n-k+1)\binom{-k}{n}\pmod{p^4},
\end{align}
where in the last step we noting that $\binom{2p-2k}{p-k}\eq0\pmod{p}$ for $k\in\{1,2,\ldots,(p-1)/2\}$.
Similarly, we obtain that
\begin{align}\label{key-16^n}
&\sum_{n=0}^{p-1}\f{2n+1}{(-16)^n}\sum_{k=0}^{n}\binom{2k}{k}^2\binom{2n-2k}{n-k}^2\notag\\
\eq&\sum_{k=0}^{(p-1)/2}(3k+1)\f{\binom{2k}{k}^3}{(-8)^k}+\sum_{k=1}^{(p-1)/2}\f{\binom{2p-2k}{p-k}^3}{(-16)^{p-k}}\sum_{n=0}^{k-1}(2n-2k+1)\binom{-k}{n}\pmod{p^4}.
\end{align}
Furthermore, with the help of Lemma \ref{keyid5} we have
\begin{align}\label{keydiff}
&\sum_{k=1}^{(p-1)/2}\f{\binom{2p-2k}{p-k}^3}{8^{p-k}}\sum_{n=0}^{k-1}(-2)^n(n-k+1)\binom{-k}{n}\notag\\
=&\sum_{k=1}^{(p-1)/2}\f{\binom{2p-2k}{p-k}^3}{8^{p-k}}\l((-1)^{k+1}(3k-1)-(-2)^k\sum_{n=0}^{k-1}(2n-2k+1)\binom{-k}{n}\r)\notag\\
\eq&-\sum_{k=(p+1)/2}^{p-1}(3k+1)\f{\binom{2k}{k}^3}{(-8)^k}+2\sum_{k=1}^{(p-1)/2}\f{\binom{2p-2k}{p-k}^3}{(-16)^{p-k}}\sum_{n=0}^{k-1}(2n-2k+1)\binom{-k}{n}\pmod{p^4}.
\end{align}
Combining \eqref{key8^n}--\eqref{keydiff} we arrive at
\begin{align}\label{key*}
&\sum_{n=0}^{p-1}\f{n+1}{8^n}\sum_{k=0}^{n}\binom{2k}{k}^2\binom{2n-2k}{n-k}^2\notag\\
\eq&2\sum_{n=0}^{p-1}\f{2n+1}{(-16)^n}\sum_{k=0}^{n}\binom{2k}{k}^2\binom{2n-2k}{n-k}^2-\sum_{k=0}^{p-1}(3k+1)\f{\binom{2k}{k}^3}{(-8)^k}\pmod{p^4}.
\end{align}

Substituting \eqref{CXHcon} and \eqref{den-16^n} into \eqref{key*} we obtain \eqref{den8^n}. This completes the proof.\qed

\begin{Acks}
This work is supported by the National Natural Science Foundation of China (Grant No. 11971222).
\end{Acks}

\end{document}